\documentclass[reqno]{amsart}
\usepackage[utf8]{inputenc}

\usepackage{xcolor}
\usepackage{graphicx}
\usepackage{amssymb}
\usepackage[shortlabels]{enumitem}
\usepackage{mathtools}
\usepackage[foot]{amsaddr}

\numberwithin{equation}{section}

\DeclareMathOperator{\closedspan}{\overline{\rm span^{\rule{-0.5pt}{3pt}}}}

\DeclareMathOperator{\var}{var}
\DeclareMathOperator{\cov}{cov}
\newcommand{\abs}[1]{\left\lvert#1\right\rvert}

\newcommand{\set}[1]{\left\{#1\right\}}
\newcommand{\E}{\mathsf E}
\newcommand{\R}{\mathbb R}

\newcommand{\Normal}{\mathcal N}
\let\ME=\E
\newcommand{\F}{\mathcal F}
\newcommand{\cM}{\mathcal M}

\newtheorem{theorem}{Theorem}[section]
\newtheorem{lemma}[theorem]{Lemma}

\newtheorem{proposition}[theorem]{Proposition}
\theoremstyle{remark}
\newtheorem{remark}[theorem]{Remark}
\theoremstyle{definition}
\newtheorem{definition}[theorem]{Definition}
\newtheorem{assumption}[theorem]{Assumption}

\begin{document}

\title[Entropy and alternative entropy functionals of fGn]{Entropy and alternative entropy functionals of~fractional Gaussian noise as the functions of~Hurst index}
\author{Anatoliy Malyarenko$^1$}
\email{anatoliy.malyarenko@mdu.se}
\address{$^1$ Division of Mathematics and Physics, M\"alardalen University, 721 23 V\"aster{\aa}s, Sweden}
\author{Yuliya Mishura{$^{1,2}$}}
\email{yuliyamishura@knu.ua}
\address{$^2$ Department of Probability Theory, Statistics and Actuarial Mathematics, Taras Shevchenko National University of Kyiv, 64/13, Volodymyrska Street, 01601 Kyiv, Ukraine}
\author{Kostiantyn Ralchenko$^{2,3}$}
\email{kostiantynralchenko@knu.ua}
\address{$^3$ Sydney Mathematical Research Institute, The University of Sydney, Sydney NSW 2006, Australia}
\author{Sergiy Shklyar$^2$}
\email{shklyar@univ.kiev.ua}
\thanks{The second author  was supported   by The Swedish Foundation for Strategic Research, grant Nr.  UKR22-0017.
The third author  was supported by the Sydney Mathematical Research Institute under Ukrainian Visitors Program.
The second and the third authors acknowledge that the present research is carried through within the frame and support of the ToppForsk project nr. 274410 of the Research Council of Norway with title STORM: Stochastics for Time-Space Risk Models.}

\begin{abstract}
This paper is devoted to the study of the properties of entropy as a function of the Hurst index, which corresponds to the fractional Gaussian noise. Since the entropy of the Gaussian vector depends on the determinant of the covariance matrix, and the  behavior of this determinant as a function of the Hurst index is rather difficult to study analytically at high dimensions, we also consider simple alternative entropy  functionals, whose behavior, on the one hand,  mimics the behavior of entropy and, on the other hand,  is not difficult to study. Asymptotic behavior of the normalized entropy (so called entropy rate) is also studied for the entropy and for the alternative functionals.
\end{abstract}

\subjclass{60G22, 60G10, 60G15, 94A17}

\keywords{Fractional Gaussian noise, Hurst index, entropy,  entropy functionals, entropy rate}

\maketitle

 \section{Introduction}

 The concept of entropy for a random variable  was introduced by Shannon \cite{Shannon} to characterize the
irreducible complexity of a particular sort of randomness.    By definition, for a   random variable $\xi$ with
probability density function $p_\xi(x)$, the entropy (that is sometimes called differential entropy, see e.g. \cite{MNB}) is given by the formula
\[
\mathbf{H}(\xi) = - \E \log p_\xi(\xi)
= - \int_{\R}  p_\xi(x) \log p_\xi(x)\,dx.
\]

Entropy of Gaussian vector was in detail studied in the book \cite{Stratonovich}. It is not difficult, therefore, to write formulas for the entropy of a stationary Gaussian process with discrete time. A particular, but rather important and interesting case of a stationary Gaussian process with discrete time is the fractional Gaussian noise with the Hurst index $H\in(0,1)$. On the one hand, it is not hard to produce the formula for the entropy of fractional Gaussian noise from formulas (5.4.5)--(5.4.6) in \cite{Stratonovich}.
In the present paper we provide the corresponding expression for the entropy of this process, see \eqref{eq:entropy}--\eqref{eq:matr}.

On the other hand, note that the behavior of the fractional Gaussian noise substantially depends on its Hurst parameter $H$. In particular, it has long memory property for $H\in(1/2, 1)$, and in the case $H\in(0, 1/2)$ it is the
process with short memory, see e.\,g., the book \cite{Mishura2008} and the papers \cite{ALN01,AMN00,AI04,CN05,NVV99}.
Of course, these properties are closely connected to the properties of corresponding fractional operators: fractional integrals and derivatives that convert the Wiener process into the fractional Brownian motion. The properties of these operators are the subject of thousands of books and papers, let us mention only the recent  general paper \cite{Luchko} and references therein. In our paper the properties of fractional operators will be reflected indirectly in a certain sense, through the properties of the corresponding random processes and their numerical characteristics.

However, a natural question about the behavior of the entropy of the fractional Gaussian noise as a function of  $H\in(0,1)$ has not been resolved, it has not even been raised. Apparently, the reason is the fact that the formula for the entropy of a Gaussian vector contains the determinant of the covariance matrix, and the  behavior of this determinant at high dimensions is rather difficult to study analytically whatever method is used, for example, the Cholesky decomposition  or  expansion using eigenvalues. By studying the behavior of entropy numerically, we noticed the effect that the entropy of fractional Gaussian noise increases with increasing $H$ from $0$ to $1/2$ and decreases with increasing $H$ from $1/2$ to $1$.  This is quite natural, since $H=1/2$ corresponds to the sequence of independent random variables, and therefore its entropy is the greatest. This is our main hypothesis, we confirm it analytically for small $n$ and numerically for large ones.

The paper is organized as follows.
Section~\ref{sec:2} is devoted to the behavior of the entropy of fractional Gaussian noise as a function of the Hurst parameter $H$ for fixed $n$. We start with the definition of the entropy and exact formulas for it in the case of fractional Gaussian noise. We present the entropy as a surface of $H$ and $n$ which clearly show the behavior of the determinant itself, its logarithm and, as a consequence, the entropy as the functions  of $H$ and $n$. Then we study in detail two particular cases, namely $n=2$ and $n=3$ which support analytically the hypothesis that the entropy of fractional Gaussian noise increases with increasing $H$ from $0$ to $1/2$ and decreases with increasing $H$ from $1/2$ to $1$.
In Section~\ref{sec:3} we are interested in the behavior of the entropy as $n\to\infty$. We derive the lower bounds for the entropy and for its limiting value known as entropy rate. Moreover, we give the exact formula for the entropy rate via spectral density.
In Section~\ref{sec:4} we introduce two alternative entropy functionals which depend on the elements of the covariance matrix, mimic the behavior of real entropy and, at the same time, are quite easy for analytical study.
The asymptotic behavior of the alternative functionals as $n\to\infty$ is studied in subsection \ref{sec:5}.
Auxiliary results concerning stationary Gaussian processes and their entropy are collected in the Appendix.

\section{Entropy of fractional Gaussian noise as a function of $H$}
\label{sec:2}

\subsection{Entropy of Gaussian vector}
 Recall again that the entropy of absolutely continuous random variable with probability density function $p_\xi(x)$ is defined by
\[
\mathbf{H}(\xi) = - \E \log p_\xi(\xi)
= - \int_{\R}  p_\xi(x) \log p_\xi(x)\,dx,
\]
see \cite[Eq.~(1.6.2)]{Stratonovich}.
Similarly, one can define the entropy of $n$-dimensional  absolutely continuous random vector, using the  joint density of its components.
In particular, if $n$-dimensional random vector $\xi$ has a multivariate Gaussian distribution $\Normal(\mu_n,\Sigma_n)$ with mean $\mu_n$ and covariance matrix $\Sigma_n$, then the logarithm of its density equals
\[
\log p_\xi(x) = - \frac12 (x-\mu_n)^\top\Sigma_n^{-1}(x-\mu_n)
- \frac{n}{2}\log(2\pi) -\frac12\log(\det\Sigma_n),
\quad x\in\R^n.
\]
Hence, the entropy of $\xi\sim\Normal(\mu_n,\Sigma_n)$ is given by
\begin{equation}\label{eq:entropy-normal}
\mathbf{H}(\xi)=\frac n2\left(1+\log(2\pi)\right)+\frac12\log(\det\Sigma_n).
\end{equation}
This is a well-known formula, see \cite[Theorem~8.4.1]{CoverThomas} or \cite[Eq.~(5.4.6)]{Stratonovich}.

\begin{remark}
1. We use natural logarithm $\log=\log_e$ in the definition of the entropy. Note that in the information theory (see, e.\,g., \cite{CoverThomas}) the entropy is sometimes defined using $\log_2$ instead of $\log$ (this is motivated by measurements in bits). In this case the formula \eqref{eq:entropy-normal} is written as follows:
\begin{equation*}
\overline{\mathbf{H}}(\xi)=\frac 12\log_2\bigl((2\pi e)^n \det\Sigma_n\bigr).
\end{equation*}

2. For Gaussian vectors, Stratonovich in  \cite{Stratonovich} introduced the alternative definition of the entropy, namely the entropy with respect to the measure
$\nu(d\xi_1,\dots,d\xi_n) = (2\pi e)^{-n/2}d\xi_1\dots d\xi_n$.
This approach leads to the following simplified version of~\eqref{eq:entropy-normal}:
\begin{equation}\label{eq:entropy-normal11}
  \mathbf{\widetilde{H}}(\xi)=\frac12\log(\det\Sigma_n).
\end{equation}
\end{remark}
\begin{remark}\label{remrem} As we shall see below, the behavior of both versions of entropy, $\mathbf{H}(\xi)$ and  $\mathbf{\widetilde{H}}(\xi)$, as the function of Hurst index are the same and coincides with the behavior of $\det\Sigma_n$: all of them increase in $H$ when $H$ increases from 0 to $1/2$ and decrease when $H$ increases from   $1/2$ to 1. Their behavior in $n$ is different: $\det\Sigma_n$ and consequently $\mathbf{\widetilde{H}}(\xi)$ decrease in $n$ for any fixed $H$, however, $\mathbf{H}(\xi)$ increases in $n$, due to the linear term $\frac n2\left(1+\log(2\pi)\right)$.
\end{remark}

\subsection{Fractional Gaussian noise}
Consider fractional Gaussian noise starting from zero. Let $B^H = \set{B^H_t, t\ge0}$ be a fractional Brownian motion (fBm) with Hurst index $H\in(0,1)$, i.e., a centered Gaussian process with covariance function of the form
\begin{equation}\label{eq:cov-fbm}
\E B^H_t B^H_s = \frac12\left(t^{2H} + s^{2H} - \abs{t-s}^{2H}\right).
\end{equation}
Let us consider the following discrete-time process:
\[
    G^H_k = B^H_k - B^H_{k-1},
    \quad k=1,2,3,\dots.
\]
It is well known that the process $B^H$ has stationary increments, which implies that $\set{G^H_k,k\ge1}$ is a stationary Gaussian sequence (known as \emph{fractional Gaussian noise}).
It follows from \eqref{eq:cov-fbm} that its
autocovariance function is given by
\begin{equation}\label{eq:rho_k}
\rho_0(H) = 1, \;\;
\rho_k(H) = \E G^H_1 G^H_{k+1}
    =\frac12 \left((k+1)^{2H} - 2k^{2H} + (k-1)^{2H}\right),
    \: k\ge 1.
\end{equation}
Therefore, according to \eqref{eq:entropy-normal}, the entropy of $(G^H_1, \dots G^H_n)$ equals
\begin{equation}\label{eq:entropy}
\mathbf{H}(G^H_1, \dots G^H_n)=\frac n2\left(1+\log(2\pi)\right)+\frac12\log (\det\Sigma_n(H)),
\end{equation}
where
\begin{equation}\label{eq:matr}
\Sigma_n(H)  = \cov (G^H_1, \dots G^H_n) =
\begin{pmatrix}
1          & \rho_1(H)     & \rho_2(H)     & \dots   & \rho_{n-1}(H) \\
\rho_1(H)     & 1          & \rho_1(H)     & \dots   & \rho_{n-2}(H) \\
\vdots     & \vdots     & \vdots     &     \ddots    & \vdots      &         \\
\rho_{n-1}(H) & \rho_{n-2}(H) & \rho_{n-3}(H) & \dots   & 1
\end{pmatrix}
\end{equation}
Formula  \eqref{eq:entropy-normal11} is transformed to
\[
  \mathbf{\widetilde{H}}(G^H_1, \dots G^H_n)=\frac12\log(\det\Sigma_n(H)).
\]

\begin{remark}\label{remark2}
\looseness=1 Let us mention several particular cases, when the determinant\linebreak $\det\Sigma_n(H)$ can be calculated explicitly.

Let $H=\frac12$.
Then all $\rho_k(\frac12)=0$, $k\ge1$, and $\rho_0(\frac12)=1$.
Therefore, $\det\Sigma_n(\frac12)=1$, $n\ge1$, and consequently $\log(\det\Sigma_n(\frac12))=0.$

Let $H=1$.
Then $B^H_t = \xi t$, where $\xi\sim\Normal(0,1)$.
Therefore, $G_k^H = \xi$, $k\ge0$, and $\rho_k(1)=1$, $k\ge0$.
This means that for any $n\ge2$ $\det\Sigma_n(1) = 0$, and consequently $\log(\det\Sigma_n(1))=-\infty.$
Moreover,
\[
\rho_k(H) = \frac12 \left((k+1)^{2H} + (k-1)^{2H} - 2 k^{2H} \right)
\to \frac12 \left((k+1)^{2} + (k-1)^{2} - 2 k^{2} \right) = 1,
\]
as $H\uparrow1$.

Let $H=0$.
Then the situation is a bit more involved.
Namely, in this case $B^0_t$ is a white noise of the form
$B^0_t = \frac{\xi_t-\xi_0}{\sqrt2}$,
where $\set{\xi_t, t\ge0}$ are $\Normal(0,1)$ independent random variables \cite{BMNZ17}.
Therefore
\[
\rho_0(0) = 1, \quad
\rho_1(0) = \frac12 \E(\xi_1-\xi_0)(\xi_2-\xi_1) = -\frac12
\quad\text{and}\quad
\rho_k(0)=0,\; k\ge2.
\]
Moreover,
\[
\rho_1 (H) = \frac12 \left(2^{2H} - 2\right)
\to -\frac12 = \rho_1(0)
\;\text{ and }\;
\rho_k(H) \to 0, \;
H\downarrow0,\; k\ge2.
\]
Consider
\[
\Sigma_n(0) =
\begin{pmatrix}
1 & -\frac12 & \cdots & 0 & 0 \\
-\frac12 & 1 & \cdots & 0 & 0\\
\vdots & \vdots & \ddots & \vdots & \vdots \\
0 & 0 & \cdots & 1 & -\frac12 \\
0 & 0 & \cdots & -\frac12 & 1
\end{pmatrix}
\]
Determinant $\det\Sigma_n (0)$ of this tridiagonal matrix is calculated by the formula
\begin{align*}
\det\Sigma_n (0) &= \det\Sigma_{n-1}(0) - \frac14 \det\Sigma_{n-2}(0)
= \dots
\\
&= \frac{k+1}{2^k} \det\Sigma_{n-k}(0) - \frac{k}{2^{k+1}} \det\Sigma_{n-k-1}(0),
\end{align*}
where
$\det\Sigma_0(0) = 1$, $\det\Sigma_{-1}(0) = 0$.
Therefore
\[
\det\Sigma_n(0) = \frac{n+1}{2^n}, \quad n\ge1,
\]
and consequently $\log(\det\Sigma_n(0))=\log(n+1)-n\log 2.$ Obviously, both $\det\Sigma_n(0)$ and $\log(\det\Sigma_n(0))$ decrease in $n$ and tend to zero and $-\infty$, respectively.
\end{remark}

\begin{figure}
\includegraphics[width=\textwidth]{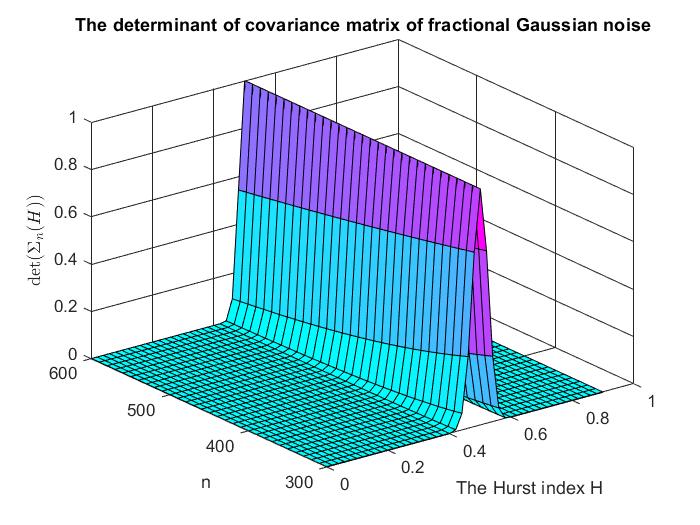}
\caption{$\det\Sigma_n(H)$ as a function  of $H$ and $n$}
\label{fig:surface}
\end{figure}

\begin{figure}
\includegraphics[width=\textwidth]{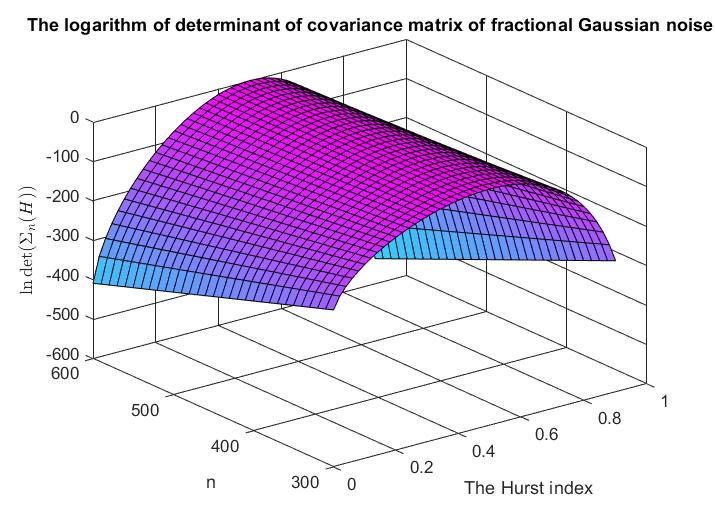}
\caption{$\log\det\Sigma_n(H)$ as a function of $H$ and $n$}
\label{fig:surface2}
\end{figure}

\begin{figure}
\includegraphics[width=\textwidth]{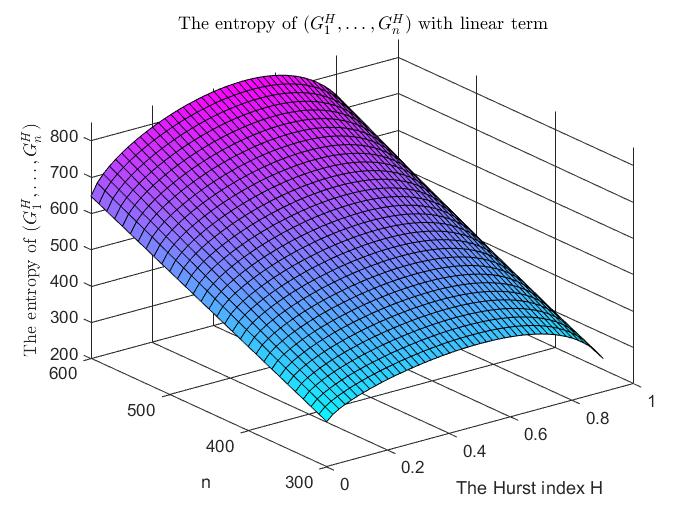}
\caption{$\mathbf{H}(G^H_1, \dots G^H_n)$ as a function of $H$ and $n$}
\label{fig:surface1}
\end{figure}

It is quite difficult to prove the monotonic properties of $\det\Sigma_n(H)$ and its logarithm analytically in general case. Therefore our main conjecture
\begin{itemize}[left=19pt]
\item[\bf(A)]
$ \det\Sigma_n(H) $ and $\log (\det\Sigma_n(H))$ increase from $\frac{n+1}{2^n }$ to 1 and from $\log(n+1)-n\log 2 $ to 0, respectively,  when $H$ increases from 0 to $\frac12$,  and decrease  from 1 to 0 and from 0 to $-\infty$, respectively  when $H$ increases from $\frac12$ to 1, decreasing in $n$ for any fixed $H$
\end{itemize}
is in general checked numerically.

The surface   of   $ \det\Sigma_n(H)$ as a function  of $H$ and $n$ is presented  at Figure~\ref{fig:surface}.
We observe that for any fixed $n\ge2$   $\det\Sigma_n(H)$ increases in $H\in(0,\frac12)$ and decreases in $H\in(\frac12,1)$. Also, it decreases in $n$ for any $H\in(0,1)$. Figures~\ref{fig:surface2} and ~\ref{fig:surface1} present  entropies $\mathbf{\widetilde{H}}(G^H_1, \dots G^H_n)$ and $\mathbf{H}(G^H_1, \dots G^H_n)$, respectively. It is more logical to arrange these entropies surfaces in this order, see Remark \ref{remrem}.

However, below we study in more detail two particular cases, namely $n=2$ and $n=3$ and prove that they increase when $H$ increases from 0 to $\frac12$ and decrease when $H$ increases from $\frac12$ to 1.  As we shall see, even in the case $n=3$ the proof of monotonicity requires a lot of technical work.

\subsection{Cases $n=2$ and $n=3$}

Consider the determinants for $n=2$ and $n=3$ in the spirit of their monotonicity in $H$.

\begin{lemma}[Case $n=2$]
The determinant $\det\Sigma_2(H)$ increases  from $\frac{3}{4}$ to 1 when $H$ increases from 0 to $\frac12$ and decreases from 1 to 0 when $H$ increases from $\frac12$ to 1. Consequently, $\log(\det\Sigma_2(H))$ increases  from $\log{3}-2\log 2$ to 0 when $H$ increases from 0 to $\frac12$ and decreases from 0 to $-\infty$ when $H$ increases from $\frac12$ to 1.
\end{lemma}
\begin{proof}
For $n=2$, we have
\begin{equation}\label{eq:detsigma2}
\det\Sigma_2(H) =
\begin{vmatrix}
1      & \rho_1(H) \\
\rho_1(H) & 1
\end{vmatrix}
= 1 - \rho_1^2(H),
\end{equation}
where
\[
\rho_1(H) = \frac12\left(2^{2H} - 2\right)
= 2^{2H-1} - 1.
\]
So,
\[
\det\Sigma_2(H) = 1 - \left(2^{2H-1} - 1\right)^2
= 1 - 2^{4H-2} + 2^{2H} - 1 = - 2^{4H-2} + 2^{2H}.
\]
Consider function
\[
\varphi_2(H) =  - 2^{4H-2} + 2^{2H}, \quad H\in(0,1).
\]
Its derivative equals
\[
\varphi_2'(H) =  - 4 \cdot 2^{4H-2} \log 2 + 2 \cdot 2^{2H} \log 2
= 2^{2H+1} \log 2 \left(1 - 2^{2H-1}\right),
\]
and $\varphi_2'(H) > 0$ for $H\in(0,\frac12)$,
$\varphi_2'(H) < 0$ for $H\in(\frac12,1)$.
\end{proof}

\begin{lemma}[Case $n=3$]
The determinant $\det\Sigma_3(H)$ increases from $\frac{1}{2}$ to 1 when $H$ increases from 0 to $\frac12$ and decreases from 1 to 0 when $H$ increases from $\frac12$ to 1. Consequently, $\log(\det\Sigma_3(H))$ increases from $-\log 2$ to 0 when $H$ increases from 0 to $\frac12$ and decreases from 0 to $-\infty$  when $H$ increases from $\frac12$ to 1.
\end{lemma}

\begin{proof}
The value of the determinant equals
\begin{equation}\label{eq:detsigma3}
\det\Sigma_3(H) =
\begin{vmatrix}
1      & \rho_1(H) & \rho_2(H) \\
\rho_1(H) & 1      & \rho_1(H) \\
\rho_2(H) & \rho_1(H) & 1
\end{vmatrix}
= 1 + 2 \rho_1^2(H) \rho_2(H) - \rho_2^2(H) - 2\rho_1^2(H),
\end{equation}
where
\[
\rho_2(H) = \frac12\left(3^{2H} - 2^{2H+1} + 1\right).
\]
Consider function
\[
\varphi_3(H) = 1 + 2 x^2 y - y^2 - 2x^2,
\;\text{ where }
x = \rho_1(H),\; y = \rho_2(H),
\]
and calculate its derivative in $H$:
\begin{align*}
\varphi_3'(H) &= 4x x'_H y - 2y y'_H - 4 x x'_H + 2 x^2 y'_H
\\
&= 2\bigl[x(2yx'_H + x y'_H) - (y y'_H +2 x x'_H)\bigr].
\end{align*}
First, let $H\in(\frac12,1]$.
Then
\[
x'_H = 2^{2H}\log2 > 0,
\quad
y'_H = 3^{2H}\log3 - 2\cdot2^{2H}\log2.
\]
Let us prove that $y'_H>0$.
Indeed,
\[
y'_H = 2^{2H+1}\log2 \left( \left(\frac32\right)^{2H}\frac{\log3}{\log4} - 1\right).
\]
If $H=\frac12$, then
\[
\left(\frac32\right)^{2H}\frac{\log3}{\log4} - 1
=\frac{\log27}{\log16} - 1 > 0.
\]
Since $y'_H$ evidently increases in $H$, it is strictly positive.
Note that $x\le1$.
Therefore for $H\in(\frac12,1]$
\[
\varphi_3'(H) < 2\bigl(2yx'_H + x y'_H - y y'_H - 2 x x'_H\bigr)
= 2(x-y)(y'_H - 2 x'_H).
\]
Further,
\[
x-y = 2^{2H-1} - 1 - \frac12\cdot 3^{2H} + 2^{2H} - \frac12
= \frac32\left( 2^{2H} - 3^{2H-1} - 1\right)
\eqqcolon \psi(H).
\]
It is easy to see that
$\psi(\frac12) = \psi(1) = 0$.
Its second derivative equals
\[
\psi''(H) = 6\left( 2^{2H}\log^2 2 - 3^{2H-1}\log^2 3\right)
= 6\cdot2^{2H}\log^2 3\left(\frac{\log^2 2}{\log^2 3} - \left(\frac32\right)^{2H}\!\cdot\frac13\right).
\]
Let $H=\frac12$.
Then
\[
\frac{\log^2 2}{\log^2 3} - \frac12
\approx \frac{0.693^2}{1.099^2} - \frac12
\approx \frac{480249}{1207801} - \frac12<0.
\]
It means that $\psi''(H) < 0$ on the interval $[\frac12,1]$.
Moreover,
\[
\psi'(\tfrac12) = 3\left(2\log2 - \log3\right)>0.
\]
It means that on the interval $[\frac12,1]$
\[
\psi(H) = x - y > 0.
\]
Let us analyze
\[
\zeta(H) = y'_H - 2 x'_H
= 3^{2H} \log 3 - 4\cdot 2^{2H}\log2
=2^{2H} \log 3 \left(\left(\frac32\right)^{2H} - \frac{\log 16}{\log 3}\right).
\]
If $H=1$, then
\[
\left(\frac32\right)^{2H} - \frac{\log 16}{\log 3}
=\frac94 - \frac{\log 16}{\log 3}
\approx \frac94 - 2.2 < 0.
\]
Consequently, $\zeta(H) < 0$, and $\varphi'_3(H) < 0$ that is equivalent to decreasing of the determinant $\det\Sigma_3(H)$ on $[1/2, 1]$.

Now, let $H\in[0,\frac12)$.
While $x'_H>0$, the situation with $y'_H$ is more involved.
Denote $H_0 = 0.2868143617175754$, the unique root of the equation
\[
3^{2H} \log 3 - 2 \cdot 2^{2H} \log 2 = 0.
\]
Then $y'_H<0$ on $[0,H_0)$ and $y'_H>0$ on $(H_0,\frac12]$. If $H \in [H_0,\frac12]$, then in the formula for $\varphi_3'(H)$ we have
\[
x \le 0, \quad y \le 0, \quad x'_H \ge 0, \quad y'_H \ge 0,
\]
whence
\[
xyx'_H \ge 0, \quad -2yy'_H \ge 0, \quad -4xx'_H \ge 0, \quad 2x^2y'_H \ge 0,
\]
i.\,e.\ $\varphi_3'(H) \ge 0$.

Now, let $H \in [0,H_0]$.
Transform $\varphi_3'(H)$ as follows:
\[
\varphi_3'(H) = 2\bigl[2xyx'_H  - y y'_H - 2 x x'_H + x^2 y'_H\bigr]
= 2\bigl[2x'_H x (y-1)  - y'_H \left(y - x^2\right)\bigr].
\]
Further, $\abs{x} < 1$, therefore $y - x^2 > y - 1$, and on $[0,H_0]$
\[
\left(-y'_H\right) \left(y - x^2\right) > \left(-y'_H\right) (y - 1).
\]
So,
\[
\varphi_3'(H) > 2(y-1) \left(2xx'_H - y'_H\right).
\]
Obviously, $y-1<0$.
Consider
\begin{align*}
2xx'_H - y'_H &= 2\left(2^{2H-1} - 1\right) \cdot 2^{2H} \log2 - 3^{2H} \log3 + 2\cdot 2^{2H} \log2
\\
&= 2^{4H} \log2 - 2\cdot 2^{2H} \log2 - 3^{2H} \log3 + 2\cdot 2^{2H} \log2
\\
&= 3^{2H} \log2 \left(\left(\frac43\right)^{2H} - \frac{\log3}{\log2}\right)
<0
\end{align*}
for any $H\in[0,H_0]$ (in fact, for any $H\in[0,\frac12]$).
Therefore, $\varphi_3'(H) > 0$  that is equivalent to increasing of the determinant $\det\Sigma_3(H)$ on $[0, 1/2]$.
\end{proof}

\begin{remark}
For all $H\in(0,1)$, $\det\Sigma_2(H)\ge\det\Sigma_3(H)$ (where the equality is achieved only for $H=\frac12$ and for $H\uparrow1$).
Indeed, by \eqref{eq:detsigma2} and \eqref{eq:detsigma3}, we get
\begin{align*}
\det\Sigma_2(H) - \det\Sigma_3(H)
&=  \rho_1^2(H)  - 2 \rho_1^2(H) \rho_2(H) + \rho_2^2(H)
\\
&=  \bigl(\rho_1(H)-\rho_2(H)\bigr)^2 + 2 \rho_1(H) \rho_2(H) \bigl(1 - \rho_1(H)\bigr) \ge 0,
\end{align*}
since $\rho_1(H)\le 1$, and $\rho_1(H)$ and $\rho_2(H)$ have the same sign (they both are negative for $H\in(0,\frac12)$ and positive for $H\in(\frac12,1)$).
Figure~\ref{fig:detplot} contains the graphs of $\det\Sigma_2(H)$ and $\det\Sigma_3(H)$.

In the general case, the monotonicity of $\det\Sigma_n(H)$ as a function of $n$ can be proved by representing it as a product of conditional variances, see Remark \ref{rem:covfgndecrn} in the appendix.
\end{remark}

\begin{figure}
\includegraphics[scale=.5]{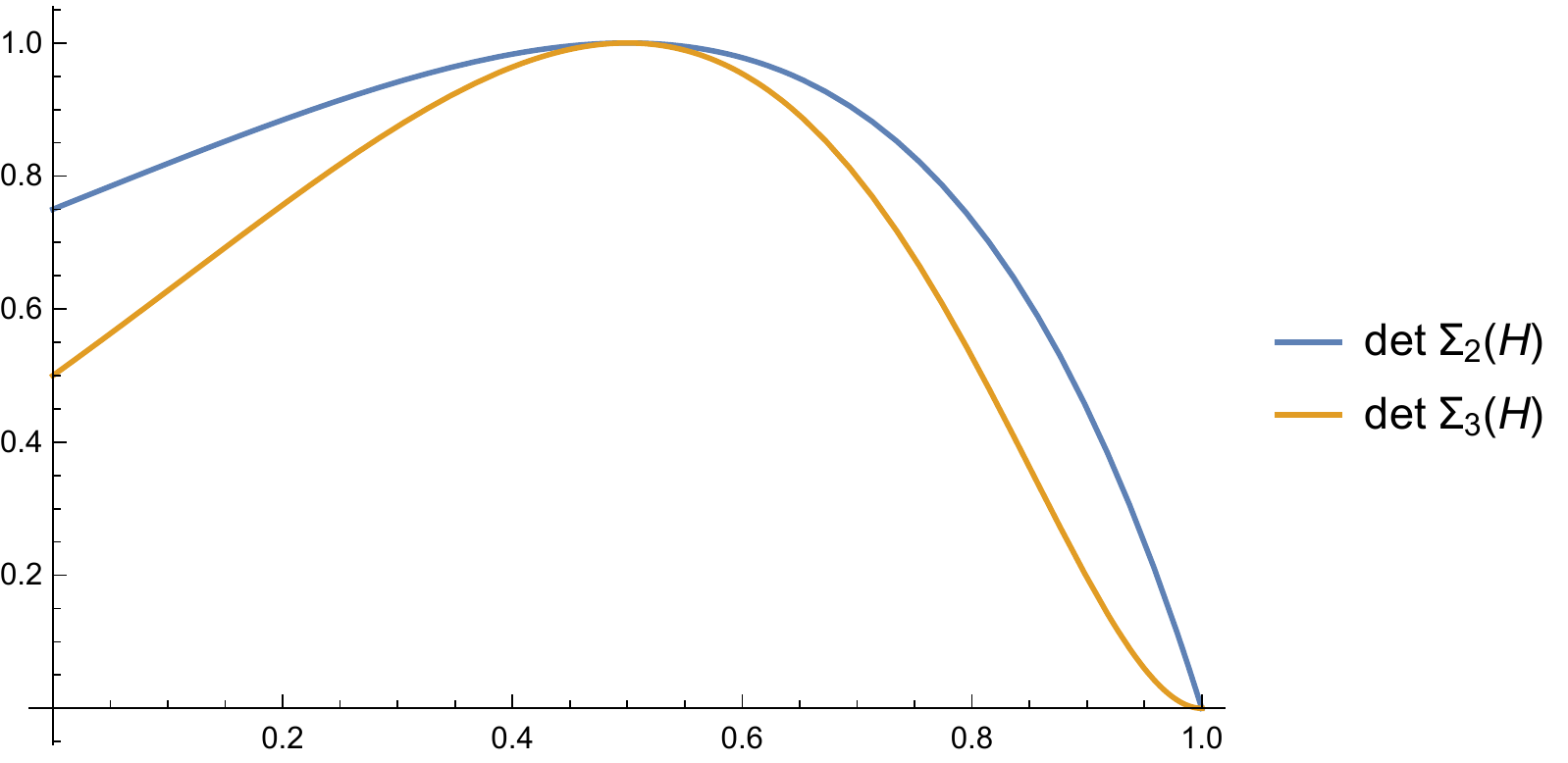}
\caption{Graphs of $\det\Sigma_2(H)$ (blue) and $\det\Sigma_3(H)$ (orange)\label{fig:detplot}}
\end{figure}

\section{Entropy, entropy rate and innovation variance.  Lower bound for innovation variance}
\label{sec:3}

\subsection{Fractional Gaussian noise on the whole axis}
\label{sec:motive_bidirectional}
Until now, we have considered the entropy of stationary fractional Gaussian noise starting from zero. However, quite often stationary processes start from $-\infty$, especially if the question of their regularity and some other properties are  being investigated. Therefore, we recall how we can construct fractional Gaussian noise starting from $-\infty$. For this purpose we use the Mandelbrot--van Ness representation of the fractional Brownian motion. Let us briefly recall the concepts related to this object.

Standard two-sided  Brownian motion   is a process $W=\{W_t, \; \allowbreak t\in\mathbb{R}\}$
constructed as  a couple of two independent Brownian motions
$\{W_{-t}, \; \allowbreak t\ge 0\}$
and
$\{W_t, \; \allowbreak t\ge 0\}$,
one with the time reflected.
Two-sided fractional Brownian motion   is   a zero-mean Gaussian process $B^H=\{B^H_t, t\in\mathbb{R}\}$ with covariance function
$$ \ME B^H_s B^H_t = \frac12 (|s|^{2H} + |t|^{2H} - |s-t|^{2H}).$$ It admits
the Mandelbrot--van Ness representation
\begin{equation}\label{eq:MVN}
B^H_t  = c_H \int_{-\infty}^t \left( (t-s)_{+}^{H-\frac12} - (-s)_{+}^{H-\frac12}\right)dW_s,
\end{equation}
where $c_H=\frac{(2H \sin(\pi H)\Gamma(2H))^{1/2}}{\Gamma(H + 1/2)}
= \left(\frac{2H \Gamma(3/2-H)}{\Gamma(H + 1/2)\Gamma(2-2H)}\right)^{1/2}$.
Obviously, process $B^H$   has  stationary increments   $B^H_s - B^H_{s-1},\,s\in\mathbb{R},$  whose  covariance equals
\begin{multline*}
\ME \left (B^H_s - B^H_{s-1}\right) \left(B^H_t - B^H_{t-1}\right )
\\*
= \frac12\left (|s-t-1|^{2H} - 2   |s-t|^{2H} + |s-t+1|^{2H}\right ), \quad s, t\in \R.
\end{multline*}

\subsection{Lower bound for the innovation variance}
According to Proposition~\ref{prop:rk} in the appendix, the entropy of a stationary Gaussian process $\set{X_k,k=1,2,\dots}$ can be expressed in terms of the following conditional variances:
\begin{equation}\label{eq:rk}
  r(k) = \var[X_k \mid X_1,\ldots,X_{k-1}],
 \end{equation}
  see formula \eqref{eq:Hviar}.
The values $r(k)$ are deterministic, nonnegative  and decreasing, hence, there exists the finite limit
\begin{equation}\label{eq:defsigma2inov}
  \sigma^2_{\rm inov}(X) = \lim_{n\to\infty} r(n) \ge 0,
\end{equation}
which is called \emph{innovation variance}.

Furthermore, for a stationary Gaussian process we have
\begin{align}
\sigma^2_{\rm inov}(X)
  &= \lim_{n\to\infty} r(n) =
  \lim_{n\to\infty} \var [X_n \mid X_{n-1},\ldots,X_1]
  \nonumber \\ &=
  \lim_{n\to\infty} \var [X_t \mid X_{t-1},\ldots,X_{t-n+1}]
  \nonumber \\ &=
  \var[X_t \mid X_{t-1}, X_{t-2}, \ldots]
  \qquad \mbox{for all $t\in\mathbb{R}$} .
  \label{eq:sigmainov_cvar}
\end{align}

It turns out that for fractional Gaussian noise $G^H$ the limit \eqref{eq:defsigma2inov} is strictly positive for all $H$, and moreover, it admits the following lower bound.

\begin{theorem}[Lower bound for the innovation variance]
\label{th:lbound}
For all $H\in(0,1)$,
\begin{equation}
  \sigma^2_{\rm inov}(G^H)
  \ge
  \frac{\Gamma\!\left(\frac32 - H\right)}
         {\Gamma\!\left(H+\frac12\right) \Gamma(2-2H)}
  \eqqcolon\sigma^2_H .
  \label{neq:sigma2inov_lb}
\end{equation}
\end{theorem}

\begin{proof}
As a particular case of \eqref{eq:sigmainov_cvar},
\[
\sigma^2_{\rm inov} \left(G^H\right) = \var\left [G^H_1 \mid G^H_0, G^H_1, G^H_{-1}, \ldots\right ] .
\]
Notice that $G^H_1 = B^H_1$, and all $G^H_t = B^H_t - B^H_{t-1}$,\allowbreak\;
$t\mathbin{\le}0$,
can be represented as integrals w.r.t.\ the Brownian motion $\{W_t, \; \allowbreak t\le 0\}$
with use of \eqref{eq:MVN},
whence
\[
 \sigma(G^H_0, G^H_{-1},G^H_{-2}, \ldots) \subset \sigma(W_s, \; s\le 0) .
\]
By the partitioning of conditional variance, see \eqref{neq:pvar2e},
\begin{equation}\label{eq:lowbou}
  \sigma^2_{\rm inov}\left(G^H\right) =
  \var[B^H_1 \mid G^H_0, G^H_{-1}, G^H_{-2}, \ldots] \ge
  \var[B^H_1 \mid W_s, \; s\le 0] .
\end{equation}
Finally, since the process $\{B^H_t,\; t>0\}$ is a Volterra Gaussian process with the representation \eqref{eq:MVN}, we see that the conditional variance in the right-hand side of \eqref{eq:lowbou} can be calculated by the formula \eqref{eq:condvarn} as follows
\[
    \var[B^H_1 \mid W_s, \; s\le 0] =
  \int_0^1 c_H^2 (1-s)^{2H-1} \, ds = \frac{c_H^2}{2H}
  = \frac{\Gamma\!\left(\frac32 - H\right)}
         {\Gamma\!\left(H+\frac12\right) \Gamma(2-2H)} .
         \qedhere
\]
\end{proof}

\subsection{Lower bound for the entropy and the entropy rate}
Taking \eqref{th:lbound} into account, let us study the asymptotic behavior of the entropy of fractional Gaussian noise as $n\to\infty$.
We start with the definition of entropy rate, see \cite[Eq.\ (4.2)]{CoverThomas}.
\begin{definition}\label{def:entropy-rate}
  The \emph{entropy rate} of a discrete-time stochastic process $X$
  is
  \[
    \mathbf{H}_\infty(X) = \lim_{n\to\infty} \frac{\mathbf H(X_1,\ldots,X_n)}{n}
  \]
  if this limit exists.
\end{definition}

For the case of Gaussian process $X$, we may define also
\[
\widetilde{\mathbf{H}}_\infty(X) = \lim_{n\to\infty} \frac{\widetilde{\mathbf H}(X_1,\ldots,X_n)}{n},
\]
where $\widetilde{\mathbf H}(X_1,\ldots,X_n)$ is introduced in \eqref{eq:entropy-normal11}.

Let $X$ be a stationary Gaussian process. Then, applying Proposition~\ref{prop:rk} from the appendix, we obtain that its entropy rate equals
\[
   \mathbf H_\infty(X)
  =
  \frac{1 + \log (2\pi)}{2}
  + \frac{1}{2} \lim_{n\to\infty}\sum_{k=1}^n \log r(k),
\]
where $r(k)$ is defined by \eqref{eq:rk}.
If $\sigma^2_{\rm inov} (X) > 0$, then
\[
\lim_{n\to\infty} \frac{1}{n} \sum_{k=1}^n \log r(k)
=
\lim_{k\to\infty} \log r(k) = \log (\sigma^2_{\rm inov}(X)),
\]
hence,
\begin{equation}\label{eq:erviainov}
 \mathbf H_\infty(X) =
\frac{1 + \log(2\pi)}{2} + \log \sigma_{\rm inov}(X).
\end{equation}
If $\sigma_{\rm inov}^2(X) = 0$, then the entropy rate of the process $X$ is infinite:
$\mathbf H_\infty(X) = -\infty$.

\medskip

Using the results of previous subsection, we can see that for the fractional Gaussian noise $G^H$, the entropy rate exists and moreover, it admits a finite lower bound. Namely, we have the following result.

\begin{theorem}[Lower bounds for the entropy and entropy rate]
The entropy and the entropy rate of fractional Gaussian noise satisfy inequalities:
\begin{gather}
   \mathbf H(G^H_1,\ldots,G^H_n)
   \ge \frac{n}{2} \Bigl( 1 + \log(2\pi) + \log \sigma^2_H\Bigr),\notag \\
   \mathbf H_\infty(G^H)
   \ge \frac{ 1 + \log(2\pi)}{2} + \log\sigma_H,
   \label{eq:lbound}
\end{gather}
where $\sigma^2_H$ is defined in \eqref{neq:sigma2inov_lb}.
\end{theorem}

\begin{proof}
Since $G^H$ is a stationary Gaussian process, we have by Proposition \ref{prop:rk},
\[
   \mathbf H(G^H_1,\ldots,G^H_n) = \frac{n + n \log(2\pi)}{2} + \frac12 \sum_{k=1}^n \log r(k)
   \ge \frac{n}{2} \Bigl( 1 + \log(2\pi) + \log \sigma^2_H\Bigr),
\]
since $r(k)\ge\sigma_{\rm inov}^2\ge \sigma^2_H$ for all $k$, see Theorem~\ref{th:lbound}.
The inequality \eqref{eq:lbound} follows immediately from the representation \eqref{eq:erviainov} and the lower bound \eqref{neq:sigma2inov_lb}.
\end{proof}

\subsection{Calculation of the entropy rate via spectral density}
According to \cite[Eq. (5.5.17)]{Stratonovich}, the entropy rate of the stationary Gaussian process $X$ can be expressed in the form
\begin{equation}\label{eq:er-spectral}
\mathbf H_\infty(X) = \frac{1 + \log(2\pi)}{2} + \frac12 \int_0^1 \log \varphi(\mu)d\mu
 = \frac{1 + \log(2\pi)}{2} + \frac12 \int_{-1/2}^{1/2} \log \varphi(\mu)d\mu,
\end{equation}
where
$\varphi(\mu) = \sum_{k=-\infty}^\infty \gamma(k) e^{-2\pi i \mu k}$.
In particular, for fractional Gaussian noise, this approach leads to the following result.

\begin{lemma}
The entropy rate of the fractional Gaussian noise admits the following representation:
\begin{multline}\label{eq:erate-spect}
\mathbf H_\infty(G^H)
 = \frac12\left(1 +\log \left (\sin(\pi H)\Gamma(2H+1)(2\pi)^{-2H}\right)\right)
\\*
{}+ \frac12 \int_{-1/2}^{1/2} \log \left (\sum_{k=-\infty}^{+\infty} |\mu+ k|^{-2H-1}\right ) d\mu.
\end{multline}
\end{lemma}

\begin{proof}
According to \cite[Proposition 2.1]{Beran} the spectral density of fractional Gaussian noise $G^H$ is given by
\begin{align*}
f(\lambda) &=
\frac{1}{2\pi} \sum_{k=-\infty}^{+\infty} \rho_k(H) e^{ik\lambda}
\\
&= \frac{1}{\pi}\sin(\pi H)\Gamma(2H+1)(1-\cos\lambda)\sum_{k=-\infty}^{+\infty} |\lambda+2\pi k|^{-2H-1}, \quad -\pi\leq \lambda\leq\pi.
\end{align*}
Therefore, it follows from \eqref{eq:er-spectral} that the entropy rate can be calculated as follows
\begin{align*}
\mathbf H_\infty(G^H) &= \frac{1 + \log(2\pi)}{2} + \frac12 \int_{-1/2}^{1/2} \log \bigl(2\pi f(2\pi\mu)\bigr) d\mu
\\
& = \frac12\left(1 +\log \left (2\sin(\pi H)\Gamma(2H+1)(2\pi)^{-2H}\right)\right)
\\
&\quad+ \frac12 \int_{-\frac12}^{\frac12} \log \bigl(1-\cos(2\pi\mu)\bigr) d\mu
+ \frac12 \int_{-1/2}^{1/2} \log \left (\sum_{k=-\infty}^{+\infty} |\mu+ k|^{-2H-1}\right ) d\mu.
\end{align*}
It is not hard to compute
$\int_{-\frac12}^{\frac12} \log \bigl(1-\cos(2\pi\mu)\bigr) d\mu = -\log 2$, whence \eqref{eq:erate-spect} follows.
\end{proof}

\begin{remark}
For computational reasons, it may be convenient to express the infinite sum from \eqref{eq:erate-spect} as
\[
\sum_{k=-\infty}^{+\infty} |\mu+ k|^{-2H-1}
= \zeta(2H+1,\mu) + \zeta(2H+1,-\mu) - |\mu|^{-2H-1}
\]
where $\zeta(s,a) = \sum_{k=0}^\infty\abs{a+k}^{-s}$ denotes the Hurwitz zeta function.
\end{remark}

Figure~\ref{fig:entrate} contains the graphs of $\frac1n \mathbf H(G_1^H,\dots,G_n^H)$ for $n = 10$, $50$, and $100$ together with the entropy rate $\mathbf H_\infty(G^H)$ (computed by the formula \eqref{eq:erate-spect}) and the lower bound \eqref{eq:lbound}.
From one hand, it confirms the convergence of the normalized entropies to the entropy rate. From the other hand, we see that formula \eqref{eq:lbound} gives rather accurate lower bound for all values of $H$.
Moreover, the graph of $\mathbf H_\infty(G^H)$ confirms the following theoretical values for particular cases (see Remark \ref{remark2}).
\begin{align*}
H&=0\colon \; \mathbf H_\infty\left (G^0\right ) = \lim\limits_{n\to\infty} \frac12 \left (1+\log\pi+\frac1n\log(n+1)\right )
 = \frac12 \left (1+\log\pi\right )\approx 1.07236;\\
H&=\tfrac12\colon \;\mathbf H_\infty\left (G^{\frac12}\right ) = \frac12 \bigl(1+\log(2\pi)\bigr)\approx 1.41894;\\
H&=1\colon \; \mathbf H_\infty\left (G^{1}\right ) = -\infty.
\end{align*}

\begin{figure}
\includegraphics{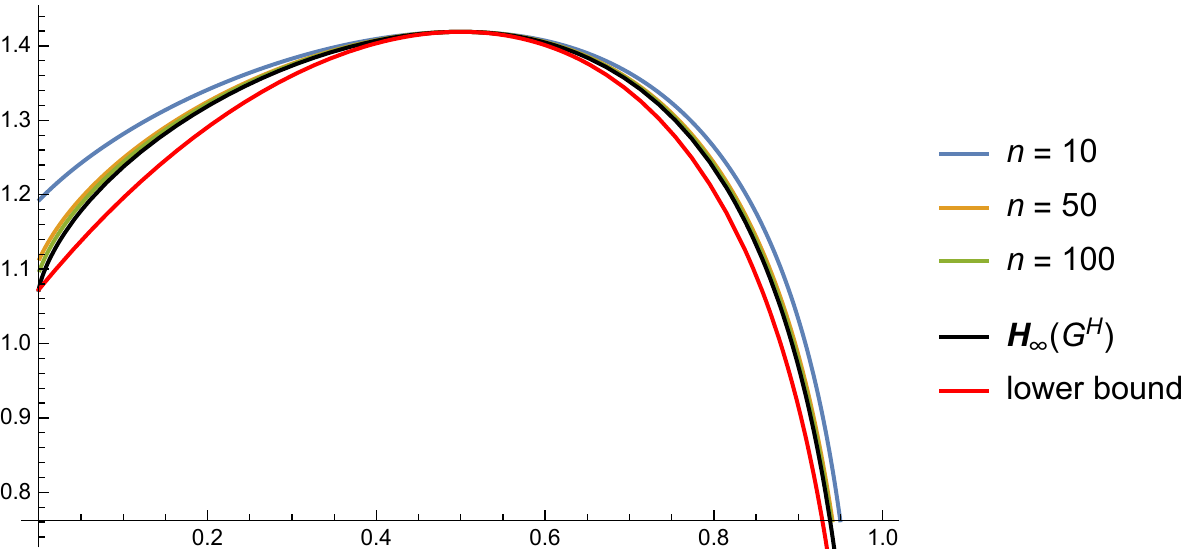}
\caption{The normalized entropy $\mathbf H(G_1^H,\dots,G_n^H)/n$ for $n = 10$, $50$, and $100$, the entropy rate $\mathbf H_\infty(G^H)$, and the lower bound~\eqref{eq:lbound}\label{fig:entrate}}
\end{figure}

\section{Entropy functionals}
\label{sec:4}
\subsection{Definition and the main properties of entropy functionals}
Taking into account two facts:
\begin{enumerate}
\item[$(i)$]
Standard entropy is related to the determinant of covariance matrix;
\item[$(ii)$]
It is impossible (or at least rather difficult) to study the properties of the determinant consequently of the entropy as the function of $H$ for the high values of  $n$,
\end{enumerate}
let us introduce two alternative entropy functionals that are based on the elements of covariance matrix in the following way: the first functional is proportional to  the sum of squares of all different elements of covariance matrix for $H\in (0,1)$:
\begin{align*}
E^1_H(N)&=-\frac{(H-1/2)^2}{1-H}F^1_H(N)  =-\frac{(H-1/2)^2}{1-H} \sum_{k=1}^N (2\rho_k(H))^2\\
&=-\frac{(H-1/2)^2}{1-H}\left(\sum_{k=2}^N \left((k+1)^{2H} + (k-1)^{2H} -2k^{2H} \right)^2
+ \left(2^{2H} -2\right)^2\right),
\end{align*}
 and the second functional is related to the permanent of  covariance matrix as follows:
\begin{align*}
E^2_H(N) &=-\frac{(H-1/2)^2}{1-H}F^2_H(N)
=-\frac{(H-1/2)^2}{1-H}\sum_{k=1}^N (N-k+1)\abs{2\rho_k(H)}\\&= -\frac{(H-1/2)^2}{1-H}\Biggl(\sum_{k=2}^N (N-k+1)\abs{(k+1)^{2H} + (k-1)^{2H} -2k^{2H}}
\\
&\quad+ N\abs{2^{2H} -2}\Biggr).
\end{align*}

\begin{remark}
 In both cases we separated the term $2^{2H} -2$ that corresponds to $k=1$ because we intend to study the behaviour of both functionals as functions of $H\in[0,1]$, and its behaviour differs from other terms. Recall also that for $H\in [1/2, 1]$ the absolute values in
$E^2_H(N)$ can be omitted. \end{remark}

\begin{theorem}
Both functionals $E^1_H(N)$ and $E^2_H(N)$ for any fixed $N\ge2$ have the following behaviour as the functions of $H\in[0,1]$: they increase in $H\in[0,\frac12]$, are zero for $H=\frac12$ and decrease in $H\in[\frac12,1]$. Functional $E^1_H(N)$ increases from $-1/4$ to 0 and decreases from 0 to $-\infty$, and $E^2_H(N)$ increases from $-N/4$ to 0 and decreases from 0 to $-\infty$.
\end{theorem}

\begin{proof} Note that the function $\phi(H)=\frac{(H-1/2)^2}{1-H}$ has a derivative $$\phi'(H)=\frac{(H-1/2)(3/2-H)}{(1-H)^2},$$
therefore it decreases on $[0,1/2]$ and increases on $[1/2,1]$ being nonnegative. Therefore it is sufficient to establish that  $F^i_H(N), i=1,2$ decrease in $H$ when $H$ increases from 0 to $1/2$ and increase in $H$ when $H$ increases from $1/2$ to 1.
First, consider $H\in(\frac12,1]$.
Then
$(k+1)^{2H} + (k-1)^{2H} -2k^{2H} > 0$
and
\begin{align*}
\frac{\partial F^1_H(N)}{\partial H}
&= 4 \sum_{k=2}^N \left((k+1)^{2H} + (k-1)^{2H} -2k^{2H} \right)
\\*
&\quad\times\left((k+1)^{2H} \log(k+1) + (k-1)^{2H}\log(k-1) -2k^{2H}\log k \right)
\\
&\quad + 4 \left(2^{2H} -2\right) 2^{2H} \log2;
\\
\frac{\partial F^2_H(N)}{\partial H}
&= 2 \sum_{k=2}^N (N-k+1)\left((k+1)^{2H} \log(k+1) + (k-1)^{2H}\log(k-1) -2k^{2H}\log k \right)
\\*
&\quad+ 2N 2^{2H} \log2.
\end{align*}
Let us analyze the value
\[
\zeta(k,H) = (k+1)^{2H} \log(k+1) + (k-1)^{2H}\log(k-1) -2k^{2H}\log k.
\]
Consider the function
\[
\varphi(x) = x^{2H} \log x,\quad x\ge1, \; 2H>1.
\]
Its second derivative equals
\begin{equation}\label{eq:phi''}
\varphi''(x) = x^{2H-2} \bigl(2H(2H-1)\log x +4H-1\bigr),
\end{equation}
and for $x\ge1$  $\varphi(x) = x^{2H} \log x>0$.
It means that $\varphi$ is convex for $x\ge1$, whence
$\zeta(k,H) >0$ for $k\ge2$, $H>\frac12$.
Obviously, both additional terms
$4 \left(2^{2H} -2\right) 2^{2H} \log2$
and
$2N 2^{2H} \log2$
are strictly positive.
So, both derivatives,
$\frac{\partial F^i_H(N)}{\partial H}>0$, $i=1,2$, $H\in(\frac12,1]$,
and so $F^1_H(N)$ and $F^2_H(N)$ are strictly increasing in $H$ from 0 to
$F^1_1(N) = 2^2 N$ and $F^2_1(N) = N (N+1)$.

\medskip

Second, consider $H\in[0,\frac12)$.
In this case
$(k+1)^{2H} + (k-1)^{2H} -2k^{2H} < 0$
for $k\ge2$, therefore, it is more convenient to rewrite
$\frac{\partial F^1_H(N)}{\partial H}$ as
\begin{equation}\label{eq:F1'}
\begin{split}
\frac{\partial F^1_H(N)}{\partial H}
&= 4 \sum_{k=2}^N \left(2k^{2H} - (k+1)^{2H} - (k-1)^{2H} \right)
\\*
&\quad\times\left(2k^{2H}\log k  - (k+1)^{2H} \log(k+1) - (k-1)^{2H}\log(k-1)\right)
\\
&\quad + 4 \log2\cdot2^{2H} \left(2^{2H} -2\right).
\end{split}
\end{equation}
Let us analyze the behaviour of all terms in \eqref{eq:F1'}.
Consider again function $\varphi$ from \eqref{eq:phi''}.
Its second derivative is negative for such $x$ that
$\log x > \frac{4H-1}{2H(1-2H)}$
and is positive if
$\log x < \frac{4H-1}{2H(1-2H)}$.
Since we consider $x\ge1$, for $H\le\frac14$ we have that $\varphi''(x)<0$ for all $x\ge1$, and for $H\in(\frac14,\frac12)$ $\varphi''(x)>0$ for $x\in(1,x_0)$ and $\varphi''(x)<0$ for $x\in(x_0,\infty)$, where
$x_0 = \exp\left(\frac{4H-1}{2H(1-2H)}\right)$.
Put $N_0 = \lfloor x_0\rfloor$.
Then
\begin{align*}
\frac{\partial F^1_H(N)}{\partial H}
&< 4 \sum_{k=N_0}^N \left(2k^{2H} - (k+1)^{2H} - (k-1)^{2H} \right)
\\*
&\quad\times\left(2k^{2H}\log k  - (k+1)^{2H} \log(k+1) - (k-1)^{2H}\log(k-1)\right)
\\*
&\quad + 4 \log2\cdot2^{2H} \left(2^{2H} -2\right).
\end{align*}
For any fixed $H\in(0,\frac12)$
$\psi(k)=2k^{2H} - (k+1)^{2H} - (k-1)^{2H} $
has a derivative
$\frac{\partial\psi}{\partial k}(k)=2H\left(2k^{2H-1} - (k+1)^{2H-1} - (k-1)^{2H-1}\right)<0$,
therefore,
\begin{align*}
\frac{\partial F^1_H(N)}{\partial H}
&< 4  \left(2N_0^{2H} - (N_0+1)^{2H} - (N_0-1)^{2H} \right)
\\
&\quad\times \sum_{k=N_0}^N
\left(2k^{2H}\log k  - (k+1)^{2H} \log(k+1) - (k-1)^{2H}\log(k-1)\right)
\\
&\quad+ 4 \log2\cdot2^{2H} \left(2^{2H} -2\right)
\\
&=4  \left(2N_0^{2H} - (N_0+1)^{2H} - (N_0-1)^{2H} \right)
\bigl(N^{2H}\log N
\\
&\quad
  - (N+1)^{2H} \log(N+1) + N_0^{2H}\log N_0  - (N_0-1)^{2H} \log(N_0-1)\bigr)
  \\
&\quad+ 4 \log2\cdot2^{2H} \left(2^{2H} -2\right)
\\
&<4  \left(2N_0^{2H} - (N_0+1)^{2H} - (N_0-1)^{2H} \right)
\bigl( N_0^{2H}\log N_0  - (N_0-1)^{2H} \log(N_0-1)\bigr)
  \\
&\quad+ 4 \log2\cdot2^{2H} \left(2^{2H} -2\right)
\\
&<4  \left(2 - 2^{2H}\right)\left (N_0^{2H}\log N_0  - (N_0-1)^{2H} \log(N_0-1)-2^{2H}\log2\right ).
\end{align*}
Again, for fixed $H$ consider function
\[
\zeta(x) = x^{2H}\log x  - (x-1)^{2H} \log(x-1),
\quad x\ge N_0.
\]
Its derivative equals
\[
\zeta'(x) = (2H\log x + 1)x^{2H-1}  - (x-1)^{2H-1}(2H\log (x-1) + 1),
\quad x\ge N_0
\]
and function
$\delta(x) = x^{2H-1}(2H\log x + 1)$
has
$\delta'(x) = \varphi''(x)<0$, $x\ge N_0$.
Therefore, $\zeta'(x)<0$, $x\ge N_0$,
and
\begin{equation}\label{eq:neg}
N_0^{2H}\log N_0  - (N_0-1)^{2H} \log(N_0-1)-2^{2H}\log2
<2^{2H}\log2-2^{2H}\log2=0.
\end{equation}

Concerning $F^2_H(N)$, for $H\in[0,\frac12)$ it equals
\[
F^2_H(N) = \sum_{k=2}^N (N-k+1)\left(2k^{2H} - (k+1)^{2H} - (k-1)^{2H}\right)
+ N\left(2-2^{2H}\right)
\]
and
\begin{align*}
\frac{\partial F^2_H(N)}{\partial H}
&= 2 \sum_{k=2}^N (N-k+1)\left(2k^{2H}\log k - (k+1)^{2H} \log(k+1) - (k-1)^{2H}\log(k-1) \right)
\\*
&\quad- 2N 2^{2H} \log2
\\
&<2N\sum_{k=N_0}^N\left(2k^{2H}\log k - (k+1)^{2H} \log(k+1) - (k-1)^{2H}\log(k-1) \right)
\\*
&\quad- 2N 2^{2H} \log2
\\
&\le2N\bigl(N^{2H}\log N - (N+1)^{2H} \log(N+1) + N_0^{2H}\log N_0
\\*
&\quad - (N_0-1)^{2H}\log(N_0-1) - 2^{2H} \log2\bigr)<0
\end{align*}
due to \eqref{eq:neg}.
\end{proof}

\subsection{Entropy rate for entropy functionals}
\label{sec:5}
It is very easy to see from formula \eqref{eq:rho_k} that $\rho_k(H)$ decrease in $k$ for $H\in(1/2,1)$ being positive and increase in
$k$ for $H\in(0,1/2)$ being negative, therefore all the summands in $(2\rho_k(H))^2$ in $E_H^1(N) $ decrease in $k$.
Moreover,
\begin{align*}
2\rho_k(H) &= k^{2H}\left(\left(1+\frac1k\right)^{2H} + \left(1-\frac1k\right)^{2H} - 2\right)
\sim 2k^{2H}\frac{2H(2H-1)}{2k^2}
\\
&= 2H(2H-1)k^{2H-2},
\quad\text{as } k\to\infty,
\end{align*}
therefore
$\bigl(2\rho_k(H)\bigr)^2 \sim 4H^2(2H-1)^2k^{4H-4}$
as $k\to\infty$.
It means that entropy functional $E_H^1(N)$ has the following asymptotic properties.

\begin{lemma}
\begin{enumerate}[(i)]
\item
Let $H\in(0,\frac34)$.
Then the series
$\sum_{k=1}^\infty \bigl(2\rho_k(H)\bigr)^2$
converges, and
\[
E_H^1(N) \to E_H^1(\infty)
= - \frac{(H-\frac12)^2}{1-H} \sum_{k=1}^\infty \bigl(2\rho_k(H)\bigr)^2
\quad\text{as } N\to\infty.
\]

\item
Let $H = \frac34$.
Then
\[
\lim_{N\to\infty} \frac{E_H^1(N)}{\log N} = -\frac{9}{16}.
\]

\item
Let $H\in(\frac34,1)$.
Then
\[
\lim_{N\to\infty} \frac{E_H^1(N)}{N^{4H-3}}
= -\frac{4H^2(2H-1)^4}{(1-H)(4H-3)}.
\]
\end{enumerate}
\end{lemma}

\begin{proof}
Item $(i)$ is evident.

$(ii)$ Indeed, with $H=3/4$
\begin{align*}
\lim_{N\to\infty} \frac{E_H^1(N)}{\log N}
&= \lim_{N\to\infty} \frac{E_{3/4}^1(N)}{\log N}
= -\frac{(H-\frac12)^2}{1-H}\lim_{N\to\infty} \frac{\bigl(2\rho_k(\frac34)\bigr)^2}{\frac1N}
\\
&= -\frac{(H-\frac12)^2}{1-H} \cdot \frac{4^2 H^2 (2H-1)^2 N^{-1}}{N^{-1}}
= - \frac{4 H^2 (2H-1)^4}{{1-H}}
= -\frac{9}{16}.
\end{align*}

$(iii)$ Indeed,
\[
\lim_{N\to\infty} \frac{E_H^1(N)}{N^{4H-3}}
= -\frac{(H-\frac12)^2}{1-H} \lim_{N\to\infty} \frac{4^2 H^2 (2H-1)^2 N^{4H-4}}{(4H-3)N^{4H-4}}
= -\frac{4H^2(2H-1)^4}{(1-H)(4H-3)}.
\qedhere
\]
\end{proof}

\begin{lemma}
\begin{enumerate}[(i)]
\item
Let $H\in(0,\frac12)$.
Then
\[
\lim_{N\to\infty} E_H^2(N)
= - \sum_{k=1}^\infty \abs{\rho_k(H)} \frac{(H-\frac12)^2}{1-H}.
\]

\item
Let $H = \frac12$.
Then $E_H^2(N) = 0$, $N\ge1$, and its limit equals zero.

\item
Let $H\in(\frac12,1)$.
Then
\[
\lim_{N\to\infty} \frac{E_H^2(N)}{N^{2H}}
= - \frac{(H-\frac12)^2}{1-H}.
\]
\end{enumerate}
\end{lemma}

\begin{proof}
Consider separately
\begin{align*}
S_N^1 &= N \sum_{k=2}^N \abs{(k+1)^{2H} + (k-1)^{2H} -2k^{2H}}
\shortintertext{and}
S_N^2 &= \sum_{k=2}^N (k-1)\abs{(k+1)^{2H} + (k-1)^{2H} -2k^{2H}}.
\end{align*}

$(i)$
Let $H\in(0,\frac12)$.
Then
$\abs{(k+1)^{2H} + (k-1)^{2H} -2k^{2H}} \sim k^{2H-2} 2H (1-2H)$,
and
$\sum_{k=2}^\infty \abs{(k+1)^{2H} + (k-1)^{2H} -2k^{2H}} < \infty$.
Therefore
\[
\frac{S^1_N}{N} \to \sum_{k=2}^\infty \abs{(k+1)^{2H} + (k-1)^{2H} -2k^{2H}},
\quad\text{as } N\to\infty.
\]
Further,
$(k-1)\abs{(k+1)^{2H} + (k-1)^{2H} -2k^{2H}} \sim k^{2H-1}2H(1-2H)$.
Therefore
\[
\lim_{N\to\infty} \frac{S^2_N}{N} \lim_{N\to\infty} N^{2H-1}2H(1-2H) = 0.
\]

$(iii)$
Let $H\in(\frac12,1)$.
Then
\[
\frac{S^1_N}{N^{2H}} \sim \frac{N^{2H-2}2H(2H-1)}{(2H-1)N^{2H-2}},
\quad\text{so}\quad
\lim_{N\to\infty} \frac{S^1_N}{N^{2H}} = 2H.
\]
Further,
\[
\frac{S^2_N}{N^{2H}} \sim \frac{N^{2H-1}2H(2H-1)}{2H N^{2H-1}},
\quad\text{so}\quad
\lim_{N\to\infty} \frac{S^2_N}{N^{2H}} = 2H-1,
\]
whence the proof follows.
\end{proof}

\appendix

\section{Some results on stationary Gaussian processes}
\label{sec:app}
\subsection{Partitioning of conditional variance}
\label{ssec:part-var}
Let $(\Omega, \mathcal{F}, \mathsf{P})$ be the probability space.
The conditional variance of the random variable $X$ with $\ME X^2<\infty$
given the $\sigma$-field $\mathcal{A}\subset\mathcal{F}$
is defined as
\begin{align*}
  \var[X \mid \mathcal{A}]
  &=
  \ME[(X - \ME[X \mid \mathcal{A}])^2 \mid \mathcal{A}]
  =
  \ME[X^2 \mid \mathcal{A}] - (\ME[X \mid \mathcal{A}])^2 .
\end{align*}

Let $\mathcal{A}$ and $\mathcal{B}$ be two $\sigma$-fields,
$\mathcal{A} \subset \mathcal{B} \subset \mathcal{F}$,
and $X$ be a random variable with $\ME X^2 < \infty$.
Then
\begin{gather}
  \var[X \mid \mathcal{A}] = \ME[\var[X \mid \mathcal{B}] \mid \mathcal{A}]
  + \var[\ME[X \mid \mathcal{B}] \mid \mathcal{A}]
  \label{eq:pvar2}
  \shortintertext{and}
  \ME \var[X \mid \mathcal{A}] = \ME \var[X \mid \mathcal{B}]
  + \ME \var[\ME[X \mid \mathcal{B}] \mid \mathcal{A}] .
  \nonumber
\end{gather}

In general case, the conditional variance $\var[X \mid \mathcal{B}]$
is a $\mathcal{B}$-measurable random variable.
In particular cases, $\var[X \mid \mathcal{B}]$ is deterministic.
For example, the conditional variance of a component of a Gaussian random vector
given other components is nonrandom.
The conditional variance of an observation $X_t$ of a Gaussian process
$X$ given observation of the process on some set $\mathrm{I}$ is nonrandom.
The same holds true for a linear functional of the process $X$.
More specifically, the following holds true:
if $X = \{X_s, \; s\in\mathrm T\}$ is a Gaussian process,
$t\in\mathrm{T}$ and $\mathrm{I}\subset\mathrm{T}$, then
$\var[X_t \mid X_s, \; s\in\mathrm{I}]$ is nonrandom.
Furthermore, if $W = \{W_s, \; s\in\mathbb{R}\}$ is a two-sided Wiener process,
$\phi \in L^2(\mathbb{R})$ is a deterministic function, and
$\mathrm{I} \subset \mathbb{R}$, then
\[
\var\left[ \int_{-\infty}^\infty \phi(s) \, dW_s \biggm| W_s, \; s\in\mathrm{I} \right]
\]
is nonrandom.

\medskip

For the Volterra Gaussian process
\[
   X_t = \int_{-\infty}^t K(t,s) \, dW_s, \qquad t\ge 0,
\]
where
$K(t,\hbox to 1ex{$\cdot$}) \in L^2((-\infty, t])$ for all $t>0$,
the following relation holds
\begin{equation}\label{eq:condvarn}
\var [X_t \mid \F_0] = \int_0^t K(t,s)^2 \, ds,
\quad t>0,
\end{equation}
where $\F_0=\sigma(W_s, s\le 0)$.
Indeed,
\[
\E(X_t\mid\F_0)
=\E\left(\int_{-\infty}^0 K(t,s) \, dW_s + \int_{0}^t K(t,s) \, dW_s
\biggm|\F_0\right) = \int_{-\infty}^0 K(t,s) \, dW_s,
\]
and similarly
\begin{align*}
\var [X_t \mid \F_0] &=
\ME[X_t^2 \mid \F_0] - (\ME[X \mid \F_0])^2
\\
&=\ME\Biggl [\left(\int_{-\infty}^0 K(t,s) \, dW_s\right)^2 + \left(\int_{0}^t K(t,s) \, dW_s
\right)^2
\\
&\quad+ 2\int_{-\infty}^0 K(t,s) \, dW_s \int_{0}^t K(t,s) \, dW_s\biggm|\F_0\Biggr] - \left(\int_{-\infty}^0 K(t,s) \, dW_s\right)^2
\\
&= \ME\left[ \left(\int_{0}^t K(t,s) \, dW_s
\right)^2\right]
= \int_0^t K(t,s)^2 \, ds.
\end{align*}
\medskip

If, in \eqref{eq:pvar2}, the conditional expectation $\var[X \mid \mathcal{B}]$
is nonrandom
(or, more generally, if $\var[X \mid \mathcal{B}]$ is an $\mathcal{A}$-measurable
random variable), then
\eqref{eq:pvar2} takes the form
\[
  \var[X \mid \mathcal{A}] = \var[X \mid \mathcal{B}]
  + \var[\ME[X \mid \mathcal{B}] \mid \mathcal{A}],
\]
whence
\begin{equation}\label{neq:pvar2e}
  \var[X \mid \mathcal{A}] \ge \var[X \mid \mathcal{B}] .
\end{equation}
The equality holds in \eqref{neq:pvar2e} if and only if
$\var[\ME[X \mid \mathcal{B}] \mid \mathcal{A}] = 0$.
The sufficient condition for equality in \eqref{neq:pvar2e}
is $\ME[X \mid \mathcal{A}] = \ME[X \mid \mathcal{B}]$
almost surely.
The sufficient conditions for strict inequality in \eqref{neq:pvar2e}
are that both $\ME[X \mid \mathcal{A}]$ and $\ME[X \mid \mathcal{B}]$ are nonrandom
and $\mathsf{P}(\ME[X \mid \mathcal{A}] \neq \ME[X \mid \mathcal{B}]) > 0$.

\subsection{Entropy of a stationary Gaussian process}
\label{apx:esgp}

Let $X = \{X_t, \; t\mathbin=1,2,\ldots\}$
be a stationary Gaussian process
with the autocovariance function
$\gamma(h) = \cov(X_{t+h},X_t)$;
The covariance matrix of $n$ consecutive
observations of the process $X$ is denoted $\Gamma_n$;
it is a symmetric Toeplitz matrix:
\[
  \Gamma_n = \begin{pmatrix}
    \gamma(0) & \gamma(1) & \ldots & \gamma(n-1) \\
    \gamma(1) & \gamma(0) & \ldots & \gamma(n-2) \\
    \hdotsfor{4} \\
    \gamma(n-1) &  \gamma(n-2) & \ldots & \gamma(0)
  \end{pmatrix}
\]

\begin{assumption}\label{as:sgp-regular}
  For all $n\in\mathbb{N}$ the matrix $\Gamma_n$ is nonsingular.
\end{assumption}

\begin{remark}
If $X$ is a fractional Gaussian noise, then Assumption~\ref{as:sgp-regular} is satisfied, see \cite[Theorem 1]{BMRS2019}.
\end{remark}

Under Assumption \ref{as:sgp-regular},
the entropy of $n$ consecutive observations of the process $X$
is equal to
\begin{equation}\label{eq:entr-gau}
  \mathbf H(X_1,\ldots,X_n) = \mathbf H(X_{t+1},\ldots,X_{t+n})
  =
  \frac{n + n\log(2 \pi)}{2} +
  \frac{1}{2} \log(\det \Gamma_n) .
\end{equation}
The goal of this subsection is to express this entropy in terms of the following quantities:
\begin{equation}\label{eq:ell2rk}
  r(1) = \var(X_1), \quad r(k) = \var[X_k \mid X_1,\ldots,X_{k-1}], \quad k=2,3,\dots
\end{equation}
Recall that $r(k)$ is nonrandom for any $k$, since the process $X$ is Gaussian, see subsection \ref{ssec:part-var}.

\begin{proposition}\label{prop:rk}
Let $X = \{X_k, \; k\mathbin=1,2,\ldots\}$
be a stationary Gaussian process, whose covariance matrix satisfies Assumption \ref{as:sgp-regular}.
Then
\begin{enumerate}[1.]
\item
The sequence
$\{r(k),\; k\in\mathbb{N}\}$, defined by \eqref{eq:ell2rk}, is deterministic, non-negative and
decreasing;
hence, it is convergent.
\item
The entropy of $n$ consecutive observations
of the process $X$ is equal to
\begin{equation}\label{eq:Hviar}
  \mathbf H(X_1,\ldots,X_n)
  =
  \frac{n + n\, \log (2\pi)}{2}
  + \frac{1}{2} \sum_{k=1}^n \log r(k).
\end{equation}
\item
The determinant of the covariance matrix $\Gamma_n$ is expressed in the following form:
\begin{equation}\label{eq:prod-inov}
\det\Gamma_n = \prod_{k=1}^n r(k).
\end{equation}
\end{enumerate}
\end{proposition}

\begin{proof}
1.  The monotonicity of $r(k)$ follows from \eqref{neq:pvar2e} and \eqref{eq:ell2rk}. Indeed,
\begin{align*}
r(k) &= \var[X_k \mid X_1,\ldots,X_{k-1}]
=
\var[X_{k+1} \mid X_2, \ldots, X_k]
\\* &\ge
\var[X_{k+1} \mid X_1, X_2, \ldots, X_k]
= r(k+1).
\end{align*}

2. Due to the chain rule for the entropies \cite[Theorem~8.6.2]{CoverThomas}
\begin{equation}\label{eq:ChainRule}
\mathbf H(X_1,\ldots,X_n) = \mathbf H (X_1) + \sum_{k=2}^n \mathbf H(X_k \mid X_1, \ldots, X_{k-1})
\end{equation}
where $\mathbf H(X_k \mid X_1,\ldots,X_{k-1})$ is the conditional entropy,
\begin{multline*}
\mathbf H(X_k \mid X_1, \ldots, X_{k-1})
   =
  - \int\dots\iint p_{X_1,\ldots,X_{k-1},X_k} (x_1,\ldots,x_{k-1},x_k)
   \\* \times
  \log p_{X_k \,|\, X_1=x_1,\ldots,X_{k-1}=x_{k-1}} (x_k)\,
  dx_1 \ldots dx_{k-1} \, dx_k,
\end{multline*}
and $\mathbf H(X_1) = \frac12 \log(2 e \pi r(1))$, since $X_1 \sim N(0,r(1))$.

The conditional distribution of $X_k$ given $X_1,\ldots,X_{k-1}$
is Gaussian, with fixed variance:
\[
  [X_k \mid X_1{=}x_1, \ldots, X_{k-1} {=} x_{k-1} ] \sim
  \mathcal{N}( \mathsf{E}[X_k \mid X_1{=}x_1, \ldots, X_{k-1}{=}x_{k-1}], \: r(k) ) .
\]
Thus, by \eqref{eq:entropy-normal}, the entropy of the conditional distribution
$[X_k \mid X_1\mathbin{=}x_1, \ldots,\allowbreak X_{k-1} \mathbin{=} x_{k-1} ]$
is
\begin{equation}\label{eq:cond-entr}
 \mathbf H(X_k \mid X_1{=}x_1, \ldots, X_{k-1} {=} x_{k-1}) = \frac12 \log(2 e \pi r(k) ).
\end{equation}
Note that the right-hand side of \eqref{eq:cond-entr} does not depend on $x_1,\ldots,x_{k-1}$.
The conditional entropy can be expressed through the entropy of the underlying conditional distribution:
\begin{multline*}
  \mathbf H(X_k \mid X_1, \ldots, X_{k-1})
   =
  \int\dots\int p_{X_1,\ldots,X_{k-1}} (x_1,\ldots,x_{k-1})
  \times \mbox{} \\ \times
  \mathbf H(X_k \mid  X_1{=}x_1,\ldots, X_{k-1}{=}x_{k-1})\,
  dx_1 \ldots dx_{k-1}.
\end{multline*}
Thus,
\begin{multline*}
 \mathbf H(X_k \mid X_1, \ldots, X_{k-1}) \\
  \begin{aligned}
   &=
  \int\dots\int p_{X_1,\ldots,X_{k-1}} (x_1,\ldots,x_{k-1})
  \,
  \frac12 \log(2 e \pi r(k))  \,
  dx_1 \ldots dx_{k-1}
   \\ &= \frac12 \log(2 e \pi r(k)) .
 \end{aligned}
\end{multline*}
By the chain rule \eqref{eq:ChainRule},
\[
\mathbf H(X_1,\ldots,X_n) = \sum_{k=1}^n \frac12 \log(2 e \pi r(k))
 = \frac{n}{2} \log(2 e \pi) + \frac12 \sum_{k=1}^n \log r(k),
\]
which coincides with \eqref{eq:Hviar}.

3. Comparing \eqref{eq:entr-gau} and  \eqref{eq:Hviar} we immediately get the representation \eqref{eq:prod-inov}.
\end{proof}

\begin{remark}
1. The first statement of Proposition~\ref{prop:rk} is known; it can be found, e.g., in \cite[Theorem 2.10.1]{Fuller}.

2. The formula \eqref{eq:prod-inov} can be proved also with the help of the Cholesky decomposition of the covariance matrix $\Gamma_n$. Namely, $\Gamma_n$ can be represented as
\[
\Gamma_n = \begin{pmatrix}
    \ell_{1,1} & 0 & \ldots & 0 \\
    \ell_{2,1} & \ell_{2,2} & \ldots & 0 \\
    \hdotsfor{4} \\
    \ell_{n,1} & \ell_{n,2} & \ldots & \ell_{n,n}
  \end{pmatrix}
  \begin{pmatrix}
    \ell_{1,1} & \ell_{2,1} & \ldots & \ell_{n,1} \\
    0 & \ell_{2,2} & \ldots & \ell_{n,2} \\
    \hdotsfor{4} \\
    0 & 0 & \ldots & \ell_{n,n}
  \end{pmatrix}
\]
where $\ell_{k,k}^2 = r(k)$, see \cite[Eq.\ (A19)]{MiRaSh}.
Hence,
$\det\Gamma_n = \prod_{k=1}^n\ell_{k,k}^2 = \prod_{k=1}^n r(k)$.
Let us mention that a similar method (based on so called $LDL^\top$-decomposition of the covariance matrix) is described in \cite[\S\,8.6]{BrockwellDavis}.
\end{remark}

\begin{remark}\label{rem:covfgndecrn}
The representation \eqref{eq:prod-inov} implies that the determinant of the covariance matrix $\Sigma_n(H)$ of the fractional Gaussian noise $G^H$ decreases as a function of $n$, since in this case
\[
r(k) \le r(1) = \var \left(G_1^H\right)= 1,
\]
for all $k$.
\end{remark}

\subsection{Entropy rate and the nondeterminism  of stationary process}

Denote by
\[
\cM_t(X)=\closedspan(X_t, X_{t-1}, X_{t-2}, \ldots)
\]
 the smallest closed linear subspace of the Hilbert space
$L^2(\Omega, \mathcal{F}, \mathsf{P})$
that contains random variables $X_t,,X_{t-1},X_{t-2},\ldots$
Let
\[
\cM_{-\infty}(X) = \bigcap_{t=-\infty}^{\infty}\cM_t(X),\qquad
\cM(X)=\closedspan(X_t, t\in\mathbb Z).
\]

\begin{definition} \label{def:wss-deterministic}
  A centered wide-sense stationary process $\{X_t, \; t\in\mathbb{Z}\}$
  is called \textit{deterministic} if
$\cM_{-\infty}(X) = \cM(X)$, i.\,e., $\cM_s(X)=\cM_t(X)$ for all $s,t\in\mathbb{Z}$.
The process $X$ is called \textit{completely non-deterministic} if $\cM_{-\infty}(X) = 0$.
 \end{definition}

\begin{lemma} \label{def:sG-deterministic}
  A centered stationary Gaussian process $\{X_t, \; t\in\mathbb{Z}\}$
  is deterministic if and only if
  \[
    \var[X_t \mid X_{t-1}, X_{t-2}, X_{t-3}, \ldots] = 0
  \]
  (here the left-hand side does not depend on $t$ due to stationarity).
\end{lemma}

It is well known that a stationary mean-zero process $X=\set{X(t),t\in\mathbb Z}$ admits the \emph{Wold's representation} as a sum of two orthogonal processes: $X(t) = M(t) + N(t)$,
where $M =\set{M(t),t\in\mathbb Z}$ is deterministic and $N=\set{N(t),t\in\mathbb Z}$ is completely non-deterministic, see, e.\,g. \cite[Appendix B.4]{ShumSt} or \cite[Section 7.1]{Bierens}.

In view of \eqref{eq:sigmainov_cvar}, a stationary Gaussian process $X$ is deterministic if and only if for this process
$\sigma_{\rm inov}^2 (X) = 0$.
Taking into account \eqref{eq:erviainov}, we get the following result.
\begin{proposition}
  Under Assumption \ref{as:sgp-regular},
  a stationary Gaussian process has a finite
  entropy rate if and only if it is non-deterministic.
\end{proposition}

\bibliographystyle{abbrv}
\bibliography{biblio}
\end{document}